\documentclass[11pt,a4paper]{amsart}
\usepackage{amssymb,amsmath,layout}

\theoremstyle{plain}
\newtheorem{thrm}{Theorem}[section]
\newtheorem{lemma}[thrm]{Lemma}

\newtheorem{rmrk}[thrm]{Remark}
\newtheorem{dfn}[thrm]{Definition}

\newtheorem*{hyp}{HYPOTHESIS}
\begin{document}

\newcommand{\SL}{\mathcal L^{1,p}( D)}
\newcommand{\Lp}{L^p( Dega)}
\newcommand{\CO}{C^\infty_0( \Omega)}
\newcommand{\Rn}{\mathbb R^n}
\newcommand{\Rm}{\mathbb R^m}
\newcommand{\R}{\mathbb R}
\newcommand{\Om}{\Omega}
\newcommand{\Hn}{\mathbb H^n}
\newcommand{\aB}{\alpha B}
\newcommand{\eps}{\ve}
\newcommand{\BVX}{BV_X(\Omega)}
\newcommand{\p}{\partial}
\newcommand{\IO}{\int_\Omega}
\newcommand{\bG}{\boldsymbol{G}}
\newcommand{\bg}{\mathfrak g}
\newcommand{\bz}{\mathfrak z}
\newcommand{\bv}{\mathfrak v}
\newcommand{\Bux}{\mbox{Box}}
\newcommand{\e}{\ve}
\newcommand{\X}{\mathcal X}
\newcommand{\Y}{\mathcal Y}
\newcommand{\W}{\mathcal W}
\newcommand{\la}{\lambda}
\newcommand{\vf}{\varphi}
\newcommand{\rhh}{|\nabla_H \rho|}
\newcommand{\Ba}{\mathcal{B}_\beta}
\newcommand{\Za}{Z_\beta}
\newcommand{\ra}{\rho_\beta}
\newcommand{\na}{\nabla_\beta}
\newcommand{\vt}{\vartheta}

\numberwithin{equation}{section}

\newcommand{\RN} {\mathbb{R}^N}
\newcommand{\Sob}{S^{1,p}(\Omega)}
\newcommand{\Dxk}{\frac{\partial}{\partial x_k}}
\newcommand{\Co}{C^\infty_0(\Omega)}
\newcommand{\Je}{J_\ve}
\newcommand{\beq}{\begin{equation}}
\newcommand{\bea}[1]{\begin{array}{#1} }
\newcommand{\eeq}{ \end{equation}}
\newcommand{\ea}{ \end{array}}
\newcommand{\eh}{\ve h}
\newcommand{\Dxi}{\frac{\partial}{\partial x_{i}}}
\newcommand{\Dyi}{\frac{\partial}{\partial y_{i}}}
\newcommand{\Dt}{\frac{\partial}{\partial t}}
\newcommand{\aBa}{(\alpha+1)B}
\newcommand{\GF}{\psi^{1+\frac{1}{2\alpha}}}
\newcommand{\GS}{\psi^{\frac12}}
\newcommand{\HFF}{\frac{\psi}{\rho}}
\newcommand{\HSS}{\frac{\psi}{\rho}}
\newcommand{\HFS}{\rho\psi^{\frac12-\frac{1}{2\alpha}}}
\newcommand{\HSF}{\frac{\psi^{\frac32+\frac{1}{2\alpha}}}{\rho}}
\newcommand{\AF}{\rho}
\newcommand{\AR}{\rho{\psi}^{\frac{1}{2}+\frac{1}{2\alpha}}}
\newcommand{\PF}{\alpha\frac{\psi}{|x|}}
\newcommand{\PS}{\alpha\frac{\psi}{\rho}}
\newcommand{\ds}{\displaystyle}
\newcommand{\Zt}{{\mathcal Z}^{t}}
\newcommand{\XPSI}{2\alpha\psi \begin{pmatrix} \frac{x}{|x|^2}\\ 0 \end{pmatrix} - 2\alpha\frac{{\psi}^2}{\rho^2}\begin{pmatrix} x \\ (\alpha +1)|x|^{-\alpha}y \end{pmatrix}}
\newcommand{\Z}{ \begin{pmatrix} x \\ (\alpha + 1)|x|^{-\alpha}y \end{pmatrix} }
\newcommand{\ZZ}{ \begin{pmatrix} xx^{t} & (\alpha + 1)|x|^{-\alpha}x y^{t}\\
     (\alpha + 1)|x|^{-\alpha}x^{t} y &   (\alpha + 1)^2  |x|^{-2\alpha}yy^{t}\end{pmatrix}}
\newcommand{\norm}[1]{\lVert#1 \rVert}
\newcommand{\ve}{\varepsilon}

\title[Quantitative uniqueness for  for zero-order perturbations, etc.]{Quantitative uniqueness for zero-order perturbations of generalized Baouendi-Grushin operators}

\author{Agnid Banerjee}
\address{Department of Mathematics\\University of California, Irvine\\
CA 92697} \email[Agnid Banerjee]{agnidban@gmail.com}

\author{Nicola Garofalo}
\address{Dipartimento di Ingegneria Civile, Edile e Ambientale (DICEA) \\ Universit\`a di Padova\\ 35131 Padova, ITALY}
\email[Nicola Garofalo]{rembdrandt54@gmail.com}

\thanks{Second author supported in part by a grant ``Progetti d'Ateneo, 2014,'' University of Padova.}

\dedicatory{Dedicated to Giovanni Alessandrini, on his $60$-th birthday, with great affection and admiration}

%
%
%

\maketitle
\begin{abstract}
Based on a variant of the frequency function approach of  Almgren, we  establish  an optimal  bound on the  vanishing order of solutions to stationary Schr\"odinger equations associated to a class of subelliptic equations with variable coefficients whose model is the so-called Baouendi-Grushin operator. Such bound provides a quantitative form of strong unique continuation that can be thought of as an analogue of the recent results of Bakri and Zhu for the standard Laplacian.
\end{abstract}

\section{Introduction}\label{S:intro}

In this note we study  quantitative uniqueness  for zero-order perturbations of variable coefficient subelliptic equations  
whose ``constant coefficient" model is the
so called Baouendi-Grushin operator.  Precisely, in $\R^N$, with $N=m+k$, we analyze equations of the form
\begin{equation}\label{e0}
\sum_{i=1}^N X_i (a_{ij}(z,t) X_j u)= V(z,t) u,
\end{equation}
where  $z\in \R^m$, $t\in \R^k$, and the vector fields $X_1,...,X_N$ are given by
\begin{align}\label{df}
& X_i= \partial_{z_i},\ \ \  i=1, ...m,\ \ \ \  \ \ \ \ X_{m+j}= |z|^{\beta} \partial_{t_j},\ \ \   j=1, ...k,\ \ \beta>0.
\end{align}
 Besides ellipticity, the $N\times N$ matrix-valued function $A(z,t) = [a_{ij}(z,t)]$ is requested to satisfy certain structural  hypothesis that will be specified in \eqref{H} in Section \ref{pre} below. These assumptions reduce to the standard Lipschitz continuity when the dimension $k = 0$, or the parameter $\beta \to 0$. The assumptions on the potential function $V(z,t)$ are specified in \eqref{vasump} below. They represent the counterpart, with respect to the non-isotropic dilations associated with the vector fields $X_1,...,X_N$, of the requirements
 \begin{equation}\label{Vu}
 |V(x)|\le M,\ \ \ \ \ \ \ |<x,DV(x)>| \le M,
 \end{equation}
for the classical Schr\"odinger equation $\Delta u = V u$ in $\Rn$. To put this paper in the proper historical perspective we recall that for this operator, and under the hypothesis \eqref{Vu}, quantitative unique continuation results akin to our have been recently obtained  in \cite{Bk}, by Carleman estimates, and in \cite{Zhu1}, by means of a variant of Almgren's frequency function introduced in \cite{Ku2}. In these papers the authors established sharp estimates on the order of vanishing of solution to Schr\"odinger equations which generalized those in \cite{DF1} and \cite{DF2} for eigenvalues of the Laplacian on a compact manifold. Our results should be seen as a generalization of those in \cite{Bk} and \cite{Zhu1} to subelliptic equations such as \eqref{e0} above. As the reader will realize such generalization is made possible by the combination of several quite non-trivial geometric facts that beautifully combine. Some of these facts are based on the previous work \cite{GV}. We also mention that  the frequency  approach in \cite{Ku2} and \cite{Zhu1} has been recently extended in \cite{BG} to obtain sharp quantitative estimates at the boundary of Dini domains for more general elliptic equations with Lipschitz principal part.
 
 When in \eqref{e0} we take $[a_{ij}]= I_N$, the identity matrix in $\R^N$, then the operator in the left-hand side of \eqref{e0} reduces  to the well known  Baouendi-Grushin operator
\begin{equation}\label{pbeta}
\mathcal B_{\beta} u= \sum_{i=1}^N X_i^2 u = \Delta_z u + |z|^{2\beta} \Delta_t u,
\end{equation}
which is degenerate elliptic along the $k$-dimensional subspace $M = \{0\}\times \R^k$.
We observe that $\mathcal B_\beta$ is not translation invariant in $\R^N$. However, it is invariant with respect to the translations along $M$.
When $\beta = 1$ the operator $\mathcal B_\beta$ is intimately connected to the sub-Laplacians in groups of Heisenberg type. In such Lie groups, in fact, in the exponential coordinates with respect to a fixed orthonormal basis of the Lie algebra the sub-Laplacian  is given by
\begin{equation}\label{slH}
\Delta_H = \Delta_z + \frac{|z|^2}{4} \Delta_t  + \sum_{\ell = 1}^k \p_{t_\ell} \sum_{i<j} b^\ell_{ij} (z_i \p_{z_j} -  z_j \p_{z_i}),
\end{equation}
where $b^\ell_{ij}$ indicate the group constants. If $u$ is a solution of $\Delta_H$ that further annihilates the symplectic vector field $\sum_{\ell = 1}^k \p_{t_\ell} \sum_{i<j} b^\ell_{ij} (z_i \p_{z_j} -  z_j \p_{z_i})$, then we see that, in particular, $u$ solves (up to a normalization factor of $4$) the operator $\mathcal B_\beta$ obtained by letting $\beta = 1$ in \eqref{pbeta} above. 

We recall that a more general class of operators modeled on $\Ba$ was first introduced by Baouendi, who studied the Dirichlet problem in weighted Sobolev spaces in \cite{Ba}. Subsequently, Grushin in \cite{Gr1}, \cite{Gr2} studied the hypoelliptcity of the operator $\Ba$ when $\beta \in \mathbb{N}$, and showed that this property is drastically affected  by addition of lower order terms.

In the paper \cite{G} the first named author introduced a frequency function associated with $\mathcal B_\beta$, and proved that such frequency is monotone nondecreasing on solutions of $\mathcal B_\beta u = 0$. Such result, which generalized Almgren's in \cite{Al}, was used to establish the strong unique continuation property for $\Ba$. The results in \cite{G} were extended to more general equations of the form \eqref{e0} by the second named author and Vassilev in \cite{GV}, following the circle of ideas in the works \cite{GL1}, \cite{GL2}.  We mention that a version of the Almgren type monotonicity formula for $\mathcal B_\beta$ played an extensive role also in the recent work \cite{CSS} on the obstacle problem for the fractional Laplacian. Remarkably, the operator $\Ba$ also played an important role in the recent work \cite{KPS} on the higher regularity of the free boundary in the classical Signorini problem. 

We can now state our main result.

\begin{thrm}\label{main}
Let $u$ be a solution to \eqref{e0} in $B_{10}$ such that $(a_{ij})$ satisfy \eqref{H} and $V$ satisfy \eqref{vasump} below. We furthermore assume that $X_{i}X_{j}u \in L^{2}_{loc}(B_{10})$ and  $|u| \leq C_0$.  Then, there exist a universal  $a\in (0,1/3)$,  depending only on $R_1, \Lambda$ in \eqref{H},  and constants $C_1,C_2$ depending on $m, k, \beta, \la, \Lambda,  C_0$ and $\int_{B_{\frac{R_1}{3}}}  u^2 \psi$, such that for all $0<r< a R_1$ one has
\begin{equation}\label{main1}
||u||_{L^{\infty}(B_r)} \geq C_1 \left(\frac{r}{R_1}\right)^{C_2 \sqrt{K}}.   
\end{equation}
\end{thrm}

It is worth emphasizing that, when $k=0$, we have $N = m$ and then \eqref{psi} below gives $\psi \equiv 1$. In such a case the constant $K$ in \eqref{vasump} below can be taken to be  $  C(||V||_{W^{1,\infty}}+1)$ for some universal $C$. We thus see that Theorem \ref{main}, when $A \equiv I_N$, reduces to the cited Euclidean result in \cite{Bk} and \cite{Zhu1}.  Therefore, Theorem \ref{main} can be thought of as a subelliptic generalization of this sharp quantitative uniqueness result for the standard Laplacian. We also would like to mention that, to the best of our knowledge, Theorem \ref{main} is new even for $\Ba u=Vu$ where $\Ba$ is as in \eqref{pbeta}.

The present paper is organized as follows. In Section \ref{pre} we introduce the  basic notations and gather some crucial preliminary results from \cite{G} and \cite{GV}. In Section \ref{S:mono} we establish a monotonicity theorem for a generalized frequency. Such result plays a central role in this paper. In Section \ref{S:prf}, we finally prove our main result, Theorem \ref{main} above.

\section{Notations and preliminary results}\label{pre}

Henceforth in this paper we follow the notations adopted in \cite{G} and \cite{GV}, with one notable proviso: the parameter $\beta>0$ in \eqref{df}, \eqref{pbeta}, etc. in this paper plays the role of $\alpha >0$ in \cite{G} and \cite{GV}. The reason for this is that we have reserved the greek letter $\alpha$ for the powers of the weight $(r^2 - \rho)^\alpha$ in definitions \eqref{h1}, \eqref{II} and \eqref{N} below.  Let $\{X_i\}$ for $i=1, ...N$ be defined as in \eqref{df}.  We denote an arbitrary point in $\RN$ as  $(z,t) \in \R^m \times \R^k$. Given a function $f$, we denote
\begin{equation}
Xf= (X_1f, .....X_Nf),\ \ \ \ \ \ \ \ \ |Xf|^2= \sum_{i=1}^N (X_i f)^2,
\end{equation}
respectively the intrinsic gradient and the square of its length. We recall from \cite{G} that the following family of anisotropic dilations are associated with the vector fields in \eqref{df} 
\begin{equation}\label{dil}
\delta_a(z,t)=(a z,a^{\beta+1} t),\ \ \ \ \ \ \ \ a>0.
\end{equation}
Let 
\begin{equation}
Q= m + (\beta+1) k.
\end{equation}
Since denoting by $dzdt$ Lebesgue measure in $\R^N$ we have $d(\delta_a(z,t)) = a^Q dz dt$, the number $Q$ plays the role of a dimension in the analysis of the operator
$\Ba$. For instance, one has the following remarkable fact (see \cite{G}) that the fundamental solution $\Gamma$ of $\Ba$  with pole at the origin is given by the formula
\[
\Gamma(z,t) = \frac{C}{\rho(z,t)^{Q-2}},\ \ \ \ \ \ \ \ \ (z,t)\not= (0,0),
\]
where $\rho$ is the pseudo-gauge 
\begin{equation}\label{rho}
\rho(z,t)=(|z|^{2(\beta+1)} + (\beta+1)^2 |t|^2)^{\frac{1}{2(\beta+1)}}.
\end{equation}
We respectively denote by 
\[
B_r = \{(z,t)\in \R^N\mid \rho(z,t) < r\},\ \ \ \ \ \ \ \ S_r = \{(z,t)\in \R^N\mid \rho(z,t) = r\}, 
\]
the gauge pseudo-ball and sphere centered at $0$ with radius $r$. The infinitesimal generator of the family of dilations \eqref{dil}  is given by the vector field
\begin{equation}\label{Z}
Z= \sum_{i=1}^m z_i \partial_{z_i} + (\beta+1)\sum_{j=1}^k t_j \partial_{y_j}.
\end{equation}
We note the important facts that
\begin{equation}\label{divZ}
\operatorname{div} Z = Q,\ \ \ \ \ \ \ \ \ \ \ \ [X_i,Z] = X_i,\ \ \ i=1,...,N.
\end{equation}
A  function $v$ is $\delta_a$-homogeneous of degree $\kappa$ if and only if $Zv=\kappa v$. Since $\rho$ in \eqref{rho} is homogeneous of degree one, we have
\begin{equation}\label{hg}
Z\rho=\rho.
\end{equation}
We also need the angle function $\psi$ introduced in \cite{G}
\begin{equation}\label{psi}
\psi = |X\rho|^2= \frac{|z|^{2\beta}}{\rho^{2\beta}}.
\end{equation}
The function $\psi$ vanishes on the characteristic manifold $M=\Rn \times \{0\}$ and clearly satisfies $0\leq \psi \leq 1$. Since $\psi$ is homogeneous of degree zero with respect to \eqref{dil}, one has
\begin{equation}\label{Zpsi}
 Z\psi = 0.
 \end{equation}

A first basic assumption on the matrix-valued function $A=[a_{ij}]$ is that it be symmetric and  uniformly elliptic. I.e., $a_{ij} = a_{ji}$, $i,j=1,...,N$, and there exists $\lambda > 0$ such that for every $(z,t)\in \R^N$ and $\eta\in \R^N$ one has
\begin{equation}\label{ue}
\lambda|\eta|^2 \leq<A(z,t)\eta, \eta> \leq \lambda^{-1} |\eta|^2.
\end{equation}
On the potential $V$ we preliminarily assume that $V\in L^\infty_{loc}(\R^N)$.
With these hypothesis in place we can introduce the notion of weak solution of \eqref{e0}.
\begin{dfn}\label{D:ws}
A weak solution to \eqref{e0} in an open set $\Om\subset \R^N$  is a function $u \in L^{2}_{loc}(\Om)$ such that the distributional horizontal gradient $Xu \in L^{2}_{loc}(\Om)$, and for which the following equality holds for all $\vf \in C^{\infty}_{0}(\Om)$
\begin{equation}\label{equation}
\int_{\Om} <AXu, X\vf> = \int_\Om V u \vf. 
\end{equation}
\end{dfn}
We note that when $A \equiv I_N$, and for a class of vector fields which are modeled on \eqref{df} above, in the pioneering paper \cite{FL} it was proved that a weak solution $u$ to \eqref{e0} is locally H\"older continuous in $\Om$ with respect to the control metric associated with the vector fields \eqref{df}.  In particular, it is continuous with respect to the Euclidean topology of $\R^N$. For the general situation of \eqref{equation} the local H\"older continuity of weak solutions can be proved essentially following \cite{FL}, but see also \cite{FGW} where such result is discussed for more general equations in the case in which $V=0$ in \eqref{equation} above. In this paper, however, all we need is the local boundedness of weak solutions of \eqref{equation}, and we do assume it a priori in Theorem \ref{main} above, so we do not need to derive it.

Throughout the paper we assume that 
\begin{equation}\label{A0}
A(0,0) = I_N,
\end{equation}
where $I_N$ indicates the identity matrix in $\R^N$.
In order to state our main assumptions (H) on the matrix $A$ it
will be useful to represent  the latter in the following block
form
\begin{equation*}
A= \begin{pmatrix} A_{11} & A_{12} \\ A_{21} & A_{22}
\end{pmatrix},
\end{equation*}
Here, the entries are respectively $ m\times
m,\thickspace m\times k,\thickspace k\times m$ and $k\times k $
matrices, and we assume that $A^t_{12}=A_{21}$.  We shall denote
by $B$ the matrix
\[
B = A - I_N,
\]
and thus
\begin{equation}\label{Bat0}
B(0,0) = O_N,
\end{equation}
thanks to \eqref{A0}.  The proof of Theorem \ref{main} relies crucially on the following assumptions on the
matrix $A$. These will be our main hypothesis and, without further mention, will be assumed to hold throughout
the paper.

\begin{hyp}
There exists a positive constant $\Lambda$ such that, for some $R_1>0$, one has in $B_{R_1}$ the
following estimates
\begin{equation*}
|b_{ij}| = |a_{ij} - \delta_{ij}|\ \leq\
\begin{cases}
\Lambda\rho, \hskip1.4truein \text{ for }\ 1\leq i,\, j\leq m,\\
\\
\Lambda \psi^{\frac12+\frac1{2\beta}}\rho\, =\, \Lambda\frac{|z|^{\beta+1}}{\rho^\beta},\quad
 \text{otherwise},
\end{cases}
\end{equation*}
\begin{equation}\label{H}
\tag{H}
\end{equation}
\begin{equation*}
|X_kb_{ij}| = |X_ka_{ij}|\ \leq\
\begin{cases}
\Lambda, \hskip1.1truein \text{ for }\quad  1\leq k\leq m ,\ \text{ and }\ 1\leq i,\,j\leq m,\\
\\
 \Lambda \GS\, =\, \Lambda \frac{|z|^{\beta}}{\rho^\beta},\quad
 \text{otherwise}.
\end{cases}
\end{equation*}
\end{hyp}

\begin{rmrk}
We note that in the situation when $k=0$  the above hypothesis coincide with the usual Lipschitz continuity at the origin of the coefficients $a_{ij}$. 
\end{rmrk}

Now we assume that  $V$ in \eqref{e0} satisfy the following hypothesis for some $K\ge 0$
\begin{equation}\label{vasump}
|V| \leq K \psi,\ \ \ \ \ \ |FV| \leq K \psi,
\end{equation}
where $\psi$ indicates the function introduced in \eqref{psi} above and $F$ is defined as in \eqref{F}. Without loss of generality we assume henceforth that $K\ge 1$. 


We next collect  several preliminary results established in \cite{GV}  that will be important in the proof of Theorem \ref{main}. We consider the quantity
\begin{equation}\label{mu}
\mu= <AX\rho, X\rho>.
\end{equation}
We note that, by the uniform ellipticity \eqref{ue} of $A$, the function $\mu$ is comparable to $\psi$ defined in \eqref{psi}, in the sense that
\begin{equation}\label{mu2}
\la \psi\le \mu \le \la^{-1} \psi.
\end{equation}
By \eqref{mu2} it is clear that, similarly to $\psi$, the function $\mu$ vanishes on the characteristic manifold $M = \{(0,t)\in \R^N\mid t\in \R^k\}$. The following vector field $F$ introduced in \cite{GV} will play an important role in this paper:
\begin{equation}\label{F}
F= \frac{\rho}{\mu} \sum_{i,j=1}^N a_{ij} X_i\rho  X_j.
\end{equation}
It is clear that $F$ is singular on $M$. However, using \eqref{fz} below and the assumptions \eqref{H} on the matrix $A$, it was shown in \cite{GV} that $F$ can be extended to all of $\R^N$ to a continuous vector field that, near the
characteristic manifold $M$, gives a small perturbation of the Euler vector
field $Z$ in \eqref{Z} above, but see also the Remark \ref{R:FZ} below.
We note from \eqref{F} that
\begin{equation}\label{Fr}
F \rho = \rho.
\end{equation}
More in general, the action of $F$ on a function $u$ is given by
\begin{equation}\label{Fu}
Fu=\frac{\rho}{\mu}<AX\rho, Xu>.
\end{equation}
We also let
\begin{equation}\label{sigma}
\sigma= <BX\rho, X\rho>= \mu - \psi.
\end{equation}
As in (2.13) in \cite{GV}, $F$ can be represented in the following way
\begin{equation}\label{fz}
F= Z- \frac{\sigma}{\mu}Z + \frac{\rho}{\mu} \sum_{i,j=1}^N b_{ij} X_i\rho X_j.
\end{equation}

\begin{rmrk}\label{R:FZ}
We emphasize that when $A(z,t) \equiv I_N$, then $B(z,t) \equiv 0_N$. In such case we immediately see from \eqref{fz} that $F \equiv Z$.
\end{rmrk}

Henceforth, for any two vector fields $U$ and $W$, $[U,W] = UW - WU$ denotes their commutator. In the next theorem we collect several important estimates that have been established in \cite{G} and \cite{GV}. 

\begin{thrm}\label{Est}
There exists a constant $C(\beta,\lambda,\Lambda,N)>0$ such that for any function $u$ one has:
\begin{itemize}
\item[(i)] $|Q- \operatorname{div} F| \leq C\rho$;
\item[(ii)] $|F\mu| \leq C \rho \psi$;
\item[(iii)] $\operatorname{div} (\frac{\sigma Z}{\mu}) \leq C \rho$;
\item[(iv)] $|X_{i}\rho|\leq \psi^{1+\frac{1}{2\beta}}, \ \ i=1,...,m,\ \ \ \ |X_{m+j} \rho| \leq (\beta+1)\rho^{1/2},\ \  j=1,...,k$;
\item[(v)] $ |F-Z| \leq C \rho^2$;
\item[(vi)] $|<FAXu, Xu>| \leq C \rho |Xu|^2$;
\item[(vii)] $|[X_i,F]u -X_iu| \leq C \rho |Xu|$,\ \ \ \ $i=1,...,N$;
\item[(viii)] $|\sigma| \leq C \rho \psi^{3/2+ \frac{1}{2\beta}}\ |X\sigma| \leq C\psi^{3/2}$;
\item[(ix)] $|\frac{b_{ij} X_j\rho X_i}{\mu}| \leq C|z|$;
\item[(x)] $|X_i\psi|\leq \frac{C\beta\psi}{|z|}, i=1,...,m,\ \ \ \ |X_{n+j} \psi| \leq \frac{C\beta\psi}{\rho}, j=1, ..., k$;
\item[(xi)] $|\frac{\sigma}{\mu}| \leq C \rho \psi,\  |Z\sigma| \leq C \rho \psi,\ |X_k\sigma| \leq C \psi^{3/2}$;
\item[(xii)] $|[X_i, -\frac{\sigma Z}{\mu}]u| \leq C \rho |Xu|$,\ \ \ \  \ (Lemma 2.7 in \cite{GV});
\item[(xiii)] $|[X_\ell, \frac{\rho}{\mu} \sum_{i,j=1}^N \frac{b_{ij} X_j\rho} X_i]u| \leq C \rho |Xu|$,\ \ $\ell = 1,...,N$.
\end{itemize}
\end{thrm}

The properties expressed in (i) and (vii) should be compared with \eqref{divZ} above.


\section{Monotonicity of a generalized frequency}\label{S:mono}

Henceforth, we denote by $u$ a weak solution to \eqref{e0} in $B_{10}$. For the sake of brevity in all the integrals involved we will routinely omit the variable of integration $(z,t)\in \R^N$, as well as Lebesgue measure $dzdt$. 
When we say that a constant is universal, we mean that it depends exclusively on $m, k, \beta$, on the ellipticity bound $\lambda$ on $A(z,t)$, see \eqref{ue} above, and on the Lipschitz bound $\Lambda$ in \eqref{H}. Likewise, we will say that $O(1)$, $O(r)$, etc. are universal if $|O(1)| \le C$, $|O(r)|\le C r$, etc., with $C\ge 0$ universal.

For $0< r <R_1$, where $R_1$ is as in the hypotheses \eqref{H} above, we define the generalized \emph{height function} of $u$ in $B_r$ as follows
\begin{equation}\label{h1}
H(r)= \int_{B_r} u^2 (r^2-\rho^2)^{\alpha} \mu, 
\end{equation}
where $\rho$ is the pseudo-gauge in \eqref{rho} above, the function $\mu$ is defined in \eqref{mu}, and $\alpha > - 1$ is going to be fixed later (precisely, in passing from \eqref{warning2} to \eqref{warning4} below).  We also introduce the \emph{generalized energy} of $u$ in $B_r$
\begin{equation}\label{II}
I(r)=  \int_{B_r}  <AXu, Xu> (r^2-\rho^2)^{\alpha+1} + \int_{B_r}  Vu^2 (r^2-\rho^2)^{\alpha+1},
\end{equation}
where, besides \eqref{ue}, the $N\times N$ matrix-valued function $A(z,t)$ fulfills the requirements \eqref{H} above, whereas the potential $V(z,t)$ satisfies the hypothesis \eqref{vasump} above. We define
the \emph{generalized frequency} of $u$ as follows
\begin{equation}\label{N}
N(r) = \frac{I(r)}{H(r)}.
\end{equation}

The central result of this section is the following monotonicity result for the frequency $N(r)$.

\begin{thrm}\label{T:mono}
There exists $R_1>0$, depending only on $R_1$ and $\Lambda$ in \eqref{H}, such that the function
\[
r \to e^{C_1r}(N(r) + C_2 K r^2),
\]
is monotone non-decreasing on the interval $(0,R_1)$. Here, $C_1$ and $C_2$ are two universal nonnegative numbers.
\end{thrm}

The proof of Theorem \ref{T:mono} will be divided into several steps.  
We begin by noting that although the gauge $\rho$ in \eqref{rho} above is not smooth at the origin, nevertheless all subsequent calculations can be  justified by integrating over the set  $B_r - B_{\ve}$, and then let $\ve \to 0$. Moreover, by standard  approximation type arguments as in \cite{GV} which crucially use the estimates in Theorem \ref{Est}, we can assume that all the computations  hereafter are classical. The initial step in the proof of Theorem \ref{T:mono}  is the following result that provides a crucial alternative representation of the generalized energy \eqref{II}.

\begin{lemma}\label{L:I}
For every $0<r<R_1$ one as
\begin{equation}\label{I}
I(r)= 2(\alpha+1) \int_{B_r} u Fu (r^2-\rho^2)^{\alpha} \mu.
\end{equation}
\end{lemma}

\begin{proof}
Using the definition of $F$, the divergence theorem and \eqref{e0}, we find
\begin{align*}
& 2(\alpha+1) \int_{B_r} u Fu  (r^2-\rho^2)^{\alpha} \mu = - \int_{B_r} u <AXu,X(r^2-\rho^2)^{\alpha+1}>
\\
& =  \int_{B_r} <AXu, Xu> (r^2-\rho^2)^{\alpha+1} +  \int_{B_r} Vu^2 (r^2-\rho^2)^{\alpha+1},
\end{align*}
which proves \eqref{I} above.

\end{proof}

\begin{lemma}[First variation formula for $H(r)$]\label{L:H'}
There exists a universal $O(1)$ such that for every $r \in (0,R_1)$ one has 
\begin{equation}\label{hr1}
H'(r)= \frac{2\alpha + Q}{r} H(r) + O(1) H(r) + \frac{1}{(\alpha+1)r} I(r).
\end{equation}
\end{lemma}

\begin{proof}
Differentiating \eqref{h1}, and using the fact that $(r^2 - \rho^2)^\alpha$ vanishes on $S_r$, we find that
\[
H'(r)= 2\alpha r \int_{B_r}  u^2 (r^2 - \rho^2)^{\alpha-1} \mu. 
\]
Using the identity
\[
(r^2 - \rho^2)^{\alpha-1} = \frac{1}{r^2} (r^2 - \rho^2)^{\alpha} + \frac{\rho^2}{r^2}(r^2 - \rho^2)^{\alpha-1},
\]
the latter equation can be rewritten as 
\[
H'(r)= \frac{2 \alpha}{r} H(r) + \frac{2\alpha}{r}\int_{B_r} u^2 (r^2-\rho^2)^{\alpha-1}\rho^2\mu.
\]
Recalling \eqref{Fr}, we have
\[
H'(r)= \frac{2\alpha}{r} H(r) - \frac{1}{r} \int_{B_r} u^2  F(r^2-\rho^2)^{\alpha}\mu.
\]
Integrating by parts, we obtain
\begin{align*}
H'(r) & = \frac{2\alpha}{r} H(r) + \frac{1}{r} \int_{B_r} \operatorname{div}(\mu u^2 F) (r^2-\rho^2)^{\alpha}
\\
& = \frac{2\alpha}{r} H(r) +  \frac{2}{r} \int_{B_r} u Fu (r^2-\rho^2)^{\alpha} \mu
\\
& + \frac{1}{r} \int_{B_r}  u^2 \operatorname{div}(F) (r^2-\rho^2)^{\alpha}\mu  + \frac{1}{r} \int_{B_r} u^2 (r^2-\rho^2)^{\alpha} F \mu.
\end{align*}
Using (i) in Theorem \ref{Est} to estimate the third term in the right-hand side, and  (ii) to estimate the forth one, we obtain 
\begin{equation}\label{hr}
H'(r)= \frac{2\alpha + Q}{r} H(r) + O(1) H(r) +  \frac{2}{r} \int_{B_r}  u Fu (r^2-\rho^2)^{\alpha} \mu.
\end{equation}

Using \eqref{I} in \eqref{hr} we conclude that \eqref{hr1} holds. 

\end{proof}

Our next result is a basic first variation formula of the generalized energy $I(r)$. Its proof will be quite laborious, and it displays many of the beautiful geometric properties of the Baouendi-Grushin vector fields \eqref{df}.  

\begin{lemma}[First variation formula for $I(r)$]\label{L:fve}
There exists a universal $O(1)$ such that for every $r\in (0,R_1)$ one has
\begin{align}\label{i'imp}
I'(r)  = \frac{2\alpha+Q}{r} I(r)  + \frac{4 (\alpha +1)}{r} \int_{B_r} (Fu)^2 (r^2- \rho^2)^{\alpha}  \mu
 +  O(1) I(r) +  O(1) K r H(r),
\end{align}
where $K\ge 1$ is the constant in \eqref{vasump}.
\end{lemma} 

\begin{proof}
Differentiating the expression \eqref{II} of $I(r)$ we obtain,
\[
I'(r)= 2(\alpha+1)r \int_{B_r} <AXu, Xu> (r^2- \rho^2)^{\alpha} + 2(\alpha+1)r \int_{B_r} Vu^2(r^2-\rho^2)^{\alpha}.
\]
Using the identity
\[
(r^2 - \rho^2)^{\alpha} = \frac{1}{r^2} (r^2 - \rho^2)^{\alpha+1} + \frac{\rho^2}{r^2}(r^2 - \rho^2)^{\alpha},
\]
we find
\begin{align}\label{I'}
I'(r) & = \frac{2(\alpha+1)}{r} \int_{B_r} <AXu, Xu>(r^2- \rho^2)^{\alpha+1} 
\\
& + \frac{2(\alpha+1)}{r}\int_{B_r} <AXu, Xu>(r^2-\rho^2)^{\alpha}\rho^2 
+ 2(\alpha+1)r \int_{B_r} Vu^2(r^2-\rho^2)^{\alpha}.
\notag
\end{align}
The second term in the right-hand side of \eqref{I'} is dealt with as follows 
\begin{align}\label{div}
\frac{2(\alpha+1)}{r}\int_{B_r} <AXu, Xu>(r^2-\rho^2)^{\alpha}\rho^2= - \frac 1r \int_{B_r} <AXu, Xu> F(r^2-\rho^2)^{\alpha+1}.
\end{align}

To compute the integral in the right-hand side of \eqref{div} we now use the following Rellich type identity in Lemma 2.11 in \cite{GV}:
\begin{align}\label{re}
& \int_{\partial B_r} <AXu,Xu><G,\nu>  = 2 \int_{\partial B_r} a_{ij}X_i u <X_j,\nu> Gu 
\\
& -  2\int_{B_r} a_{ij} (\operatorname{div} X_i) X_ju Gu  - 2 \int_{B_r} a_{ij} X_iu [X_j,G]u
\notag\\
& + \int_{B_r} \operatorname{div} G <AXu,Xu> +  \int_{B_r} <(GA)Xu,Xu>  - 2 \int_{B_r} Gu  X_i(a_{ij} X_j u),
\notag
\end{align}
where  $G$ is a vector field, $GA$ is the matrix with coefficients $Ga_{ij}$, $ \nu $ denotes  the outer unit normal to $B_r$, and the summation convention over repeated indices has been adopted.
Since for the vector fields $X_1,...,X_N$ in \eqref{df} above we have $\operatorname{div} X_i = 0$, if in \eqref{re} we take a vector field such that $G\equiv 0$ on $\partial B_r$, we obtain
\begin{align}\label{rl}
&    \int_{B_r} \operatorname{div} G <AXu, Xu> =  2 \int_{B_r} a_{ij} X_{i} u [X_j,G]u 
\\
& - \int_{B_r} <(GA) Xu, Xu> + 2 \int_{B_r} Gu X_i (a_{ij} X_j u).
\notag
\end{align}
In the identity \eqref{rl} we now take $G= (r^2- \rho^2)^{\alpha+1}F$. We remark that, while in our situation the vector fields $X_i$ and $G$ are not smooth, one can nonetheless rigorously justify the implementation of \eqref{rl} as in \cite{GV} by standard approximation arguments based on the key estimates in Theorem \ref{Est} above.
Now we look at each individual term in \eqref{rl}. We first note that from \eqref{e0} the last integral in the right-hand side of \eqref{rl} equals $-2 \int_{B_r} Fu Vu (r^2- \rho^2)^{\alpha+1}$. For the left-hand side of \eqref{rl} we have instead
\begin{align}\label{rlbis}
& \int_{B_r} \operatorname{div} G <AXu, Xu>= \int_{B_r} \operatorname{div} F <AXu, Xu> (r^2-\rho^2)^{\alpha+1}
\\
& + \int_{B_r}  <AXu, Xu> F(r^2-\rho^2)^{\alpha+1}.
\notag
\end{align}
Combining \eqref{rl} and \eqref{rlbis}, we reach the conclusion
\begin{align}\label{rlter}
& - \int_{B_r}  <AXu, Xu> F(r^2-\rho^2)^{\alpha+1} = \int_{B_r} \operatorname{div} F <AXu, Xu> (r^2-\rho^2)^{\alpha+1}
\\
&  + \int_{B_r} <(FA) Xu, Xu> (r^2- \rho^2)^{\alpha+1} - 2 \int_{B_r} a_{ij} X_{i} u [X_j, G]u
\notag
\\
&  - 2 \int_{B_r} Fu Vu (r^2- \rho^2)^{\alpha+1}.
\notag
\end{align}

Using (i) in Theorem \ref{Est} we find
\begin{align}\label{ok1}
&  \int_{B_r} \operatorname{div} F <AXu, Xu> (r^2-\rho^2)^{\alpha+1} = Q \int_{B_r}  <AXu, Xu>(r^2-\rho^2)^{\alpha+1} 
\\
& + O(r) \int_{B_r}  <AXu, Xu> (r^2-\rho^2)^{\alpha+1}.
\notag
\end{align}
Using (vi) in Theorem \ref{Est} we have
\begin{equation}\label{ok2}
\int_{B_r} <(FA) Xu, Xu> (r^2- \rho^2)^{\alpha+1} = O(r) \int_{B_r}  <AXu, Xu> (r^2-\rho^2)^{\alpha+1}.
\end{equation}
We next keep in mind that
\[
[X_j,G] = - 2(\alpha +1)\rho (r^2- \rho^2)^{\alpha}X_j \rho F + (r^2- \rho^2)^{\alpha+1} [X_j,F].
\]
This gives
\begin{align*}
a_{ij} X_i u [X_j,G] u & = - 2(\alpha +1) (r^2- \rho^2)^{\alpha} \rho <A X \rho,Xu> Fu + (r^2- \rho^2)^{\alpha+1} a_{ij} X_i u [X_i,F]u
\\
& = - 2(\alpha +1) (r^2- \rho^2)^{\alpha} (Fu)^2 \mu   + (r^2- \rho^2)^{\alpha+1} a_{ij} X_i u \left([X_j,F]u - X_j u\right)
\\
& \ \ \ + (r^2- \rho^2)^{\alpha+1} <AXu,Xu>,
\end{align*}
where we have used the fact that
\[
\rho <AX\rho,Xu> = \mu Fu,
\]
which follows from \eqref{Fu} above. We thus conclude that
\begin{align}\label{ok3}
& - 2 \int_{B_r} a_{ij} X_{i} u [X_j, G]u =   - 2 \int_{B_r}  <AXu,Xu> (r^2- \rho^2)^{\alpha+1}
\\
& + O(r) \int_{B_r}  <AXu,Xu> (r^2- \rho^2)^{\alpha+1} + 4 (\alpha +1) \int_{B_r} (Fu)^2 (r^2- \rho^2)^{\alpha}  \mu,
\notag 
\end{align}
where we have used the crucial estimate (vii) in Theorem \ref{Est} to control the integral 
\[
\int_{B_r}a_{ij} X_i u \left([X_j,F]u - X_j u\right)(r^2- \rho^2)^{\alpha+1}.
\]
Using \eqref{ok1}, \eqref{ok2} and \eqref{ok3} in \eqref{rlter}, we conclude
\begin{align}\label{rlquater}
& - \int_{B_r}  <AXu, Xu> F(r^2-\rho^2)^{\alpha+1} = (Q-2) \int_{B_r}  <AXu,Xu> (r^2- \rho^2)^{\alpha+1} 
\\
& + O(r) \int_{B_r}  <AXu,Xu> (r^2- \rho^2)^{\alpha+1} + 4 (\alpha +1) \int_{B_r} (Fu)^2 (r^2- \rho^2)^{\alpha}  \mu
\notag\\
& - 2 \int_{B_r} Fu Vu (r^2- \rho^2)^{\alpha+1}.
\notag
\end{align}
With \eqref{rlquater} in hands we now return to \eqref{div} to find
\begin{align}\label{divbis}
& \frac{2(\alpha+1)}{r}\int_{B_r} <AXu, Xu>(r^2-\rho^2)^{\alpha}\rho^2= \frac{Q-2}{r} \int_{B_r}  <AXu,Xu> (r^2- \rho^2)^{\alpha+1} 
\\
& + O(1) \int_{B_r}  <AXu,Xu> (r^2- \rho^2)^{\alpha+1} + \frac{4 (\alpha +1)}{r} \int_{B_r} (Fu)^2 (r^2- \rho^2)^{\alpha}  \mu
\notag\\
& - \frac{2}{r} \int_{B_r} Fu Vu (r^2- \rho^2)^{\alpha+1}.
\notag
\end{align}

The equation \eqref{divbis} is the central one in the proof of the first variation of the energy. Such equation allows us to unravel the second term in the right-hand side of \eqref{div} above, to which we now return to find
\begin{align*}
I'(r) & = \frac{2\alpha+Q}{r} \int_{B_r} <AXu, Xu>(r^2- \rho^2)^{\alpha+1} + \frac{4 (\alpha +1)}{r} \int_{B_r} (Fu)^2 (r^2- \rho^2)^{\alpha}  \mu
\\
&  + O(1) \int_{B_r}  <AXu,Xu> (r^2- \rho^2)^{\alpha+1} - \frac{2}{r} \int_{B_r} Fu Vu (r^2- \rho^2)^{\alpha+1}
\notag\\
&  + 2(\alpha+1)r \int_{B_r} Vu^2(r^2-\rho^2)^{\alpha}.
\notag
\end{align*}
Recalling the definition \eqref{II} of $I(r)$ we see that we can rewrite the latter equation as follows
\begin{align}\label{II'}
I'(r) & = \frac{2\alpha+Q}{r} I(r) -  \frac{2\alpha+Q}{r} \int_{B_r}  Vu^2 (r^2-\rho^2)^{\alpha+1} +  \frac{4 (\alpha +1)}{r} \int_{B_r} (Fu)^2 (r^2- \rho^2)^{\alpha}  \mu
\\
&  + O(1) I(r) - O(1) \int_{B_r}  V u^2 (r^2- \rho^2)^{\alpha+1} - \frac{2}{r} \int_{B_r} Fu Vu (r^2- \rho^2)^{\alpha+1}
\notag\\
&  + 2(\alpha+1)r \int_{B_r} Vu^2(r^2-\rho^2)^{\alpha}.
\notag
\end{align}
An integration by parts now gives
\begin{align}\label{calc}
& - \frac{2}{r} \int_{B_r} Fu Vu (r^2- \rho^2)^{\alpha+1} = - \frac{1}{r} \int_{B_r} F(u^2) V (r^2- \rho^2)^{\alpha+1}
\\
& =  \frac{1}{r} \int_{B_r} u^2 \operatorname{div}((r^2- \rho^2)^{\alpha+1} V F) = \frac{1}{r} \int_{B_r} V u^2 (r^2- \rho^2)^{\alpha+1} \operatorname{div} F 
\notag
\\
& + \frac{1}{r} \int_{B_r} u^2 FV (r^2- \rho^2)^{\alpha+1}  - \frac{2(\alpha + 1)}{r} \int_{B_r} V u^2 \rho F\rho (r^2- \rho^2)^{\alpha}.
\notag
\end{align}
Since one has trivially $(r^2- \rho^2)^{\alpha+1} \le r^2 (r^2- \rho^2)^{\alpha}$, from the assumptions \eqref{vasump} above, from \eqref{ue} and from (i) in Theorem \ref{Est}, we find
\begin{equation}\label{c1}
\left|\frac{1}{r} \int_{B_r} V u^2 (r^2- \rho^2)^{\alpha+1} \operatorname{div} F\right| \le C K r \int_{B_r} u^2 (r^2- \rho^2)^{\alpha} \mu = C K r H(r),
\end{equation}
 where $C = C(\beta,m,k,\lambda)>0$ is universal. Similarly, one has
 \begin{equation}\label{c2}
 \left|\frac{1}{r} \int_{B_r} u^2 FV (r^2- \rho^2)^{\alpha+1}\right| \le C K r H(r).
 \end{equation}
 Now  we rewrite $- \frac{2\alpha+Q}{r}\int_{B_r} Vu^2 (r^2 - \rho^2)^{\alpha+1}$ as
 \begin{align}\label{d1}
& - \frac{2\alpha+Q}{r}\int_{B_r} Vu^2 (r^2 - \rho^2)^{\alpha+1}= -2(\alpha+Q)r \int_{B_r} Vu^2 (r^2 - \rho^2)^{\alpha} 
\\
& + \frac{2 \alpha +Q}{r} \int_{B_r} Vu^2 (r^2 - \rho^2)^{\alpha} \rho^2
\notag
\end{align}
 
 Finally, since by \eqref{Fr} we have $F\rho = \rho$, we obtain
\begin{equation}\label{d2}
 \frac{2(\alpha + 1)}{r} \int_{B_r} V u^2 \rho F\rho (r^2- \rho^2)^{\alpha} = \frac{2(\alpha + 1)}{r} \int_{B_r} V u^2 \rho^2 (r^2- \rho^2)^{\alpha}
 \end{equation}
 Therefore by using \eqref{c1}, \eqref{c2}, \eqref{d1} and \eqref{d2} in \eqref{calc}, we thus conclude

\[
I'(r)  = \frac{2\alpha+Q}{r} I(r)  + \frac{4 (\alpha +1)}{r} \int_{B_r} (Fu)^2 (r^2- \rho^2)^{\alpha}  \mu
 +  O(1) I(r) + O(1) K r H(r),
\]
which is \eqref{i'imp}.

\end{proof}

We are now in a position to provide the

\begin{proof}[Proof of Theorem \ref{T:mono}]
Using \eqref{N}, and the equations \eqref{hr1} in Lemma \ref{L:H'} and \eqref{i'imp} in Lemma \ref{L:fve}, we find for some universal $C_1, C_3 \ge 0$,
\begin{align}\label{fn'}
N'(r) & = \frac{I'(r)}{H(r)} - \frac{H'(r)}{H(r)} N(r)  = O(1) N(r) + O(1) K r
\\
& +  \left(4(\alpha+1) \int_{B_r} (Fu)^2  (r^2 - \rho^2)^{\alpha} \mu
 -   \frac{1}{(\alpha+1)} \frac{I(r)^2}{H(r)}\right) \frac{1}{r H(r)}
\notag\\
& \ge - C_1 N(r) - C_3 K r,
\notag
\end{align}
where in the last inequality, we have used the fact that, in view of \eqref{I} in Lemma \ref{L:I}, the Cauchy-Schwarz inequality and the definition \eqref{H} of $H(r)$, we have
\begin{align*}
I(r)^2 & = 4(\alpha+1)^2 \left(\int_{B_r} u Fu (r^2-\rho^2)^{\alpha} \mu\right)^2
\\
& \le 4(\alpha+1)^2 H(r) \int_{B_r} (Fu)^2 (r^2-\rho^2)^{\alpha} \mu.
\end{align*} 
The inequality \eqref{fn'} implies that, with $C_2 = C_3/2$, the function
\[
r \to e^{C_1r} (N(r) + C_2 K r^2)
\]
is nondecreasing.

\end{proof}

\section{Proof of Theorem \ref{main}}\label{S:prf}

This final section is devoted to proving the main result in this paper, Theorem \ref{main}. We start from Theorem \ref{T:mono} which implies
\[
e^{C_1r} (N(r) + C_2 K r^2) \le e^{C_1 s} (N(s) + C_2 K s^2), \ \ \ \ \ \ \ \ \ \text{for}\  0<r<s<R_1.
\]
Henceforth, without loss of generality we assume that $R_1\le 1$. The latter monotonicity property implies, in particular, the existence of universal constants $C_2>0$ and $\overline C\ge 1$ such that
\begin{equation}\label{ineq}
N(r) \leq \overline C (N(s) + C_2 K),\ \ \ \ \ \ \ \ \ \text{for}\  0<r<s<R_1.
\end{equation}
Returning to \eqref{hr1} in Lemma \ref{L:H'}, we rewrite it in the following form
\begin{equation}\label{vnbis}
\frac{d}{dr} \log\left(\frac{H(r)}{r^{2 \alpha+ Q}}\right) = O(1) + \frac{1}{(\alpha +1)r} N(r),\ \ \ \ 0<r<R_1,
\end{equation}
where $|O(1)| \le C$, with $C$ universal.

Suppose now that $0<r_1 < r_2 < 2r_2 < r_3< R_1$. Integrating \eqref{vnbis} between $r_1$ and $2r_2$, and using \eqref{ineq}, we find 
\begin{equation}\label{top}
\frac{\log \frac{H(2r_2)}{H(r_1)} - C}{\log \left(\frac{2r_2}{r_1}\right)} - (2\alpha + Q)  \le   \frac{\overline C}{\alpha + 1} \left(N(2r_2) + C_2 K\right).
\end{equation}
Next, we integrate \eqref{vnbis} between $2r_2$ and $r_3$, and again using \eqref{ineq} we find 
\begin{equation}\label{bottom}
\frac{\overline C}{\alpha + 1} \left(N(2r_2) - \overline C C_2 K \right)
\le  \overline{C}^2 \left[\frac{\log \frac{H(r_3)}{H(2r_2)} + C}{
 \log\left(\frac{r_3}{2r_2}\right)}  - (2\alpha + Q)\right].
\end{equation}
Combining \eqref{top} and \eqref{bottom} we conclude
\[
\frac{\log \frac{H(2r_2)}{H(r_1)} - C}{\overline{C}^2 \log \left(\frac{2r_2}{r_1}\right)}   \le  \frac{\log \frac{H(r_3)}{H(2r_2)} + C}{
 \log\left(\frac{r_3}{2r_2}\right)} + C' \frac{K}{\alpha +1} - \left(1 - \frac{1}{\overline{C}^2}\right)(2\alpha + Q),
 \]
 where we have let $C' = (\overline C + 1)/\overline C$. Since $\overline C \ge 1$, if we now set
\[
\alpha_0 = 
 \log\left(\frac{r_3}{2r_2}\right),\ \ \ \beta_0 = \overline C^{2} \log \left(\frac{2r_2}{r_1}\right),
 \]
then  we obtain
\begin{equation}\label{warning}
\alpha_0 \log \frac{H(2r_2)}{H(r_1)} \le \beta_0  \log \frac{H(r_3)}{H(2r_2)}  + C (\alpha_0 + \beta_0) + C' \frac{K}{\alpha + 1} \alpha_0 \beta_0.
\end{equation}
Dividing both sides of the latter inequality by the quantity $\alpha_0 + \beta_0$, we find
\[
 \log \left(\frac{H(2r_2)}{H(r_1)}\right)^\frac{\alpha_0}{\alpha_0 + \beta_0} \le   \log \left(\frac{H(r_3)}{H(2r_2)}\right)^\frac{\beta_0}{\alpha_0 + \beta_0} + C  + C' \frac{K}{\alpha + 1} \frac{\alpha \beta_0}{\alpha_0 + \beta_0}.
\]
This gives
\begin{equation}\label{warning2}
\log H(2r_2) \le \log \left[H(r_3)^\frac{\beta_0}{\alpha_0 + \beta_0} H(r_1)^\frac{\alpha_0}{\alpha_0 + \beta_0}\right] + C  + C' \frac{K}{\alpha + 1} \alpha_0,
\end{equation}
 where we have used the trivial estimate $\frac{\beta_0}{\alpha_0 + \beta_0} \le 1$. Exponentiating both sides of \eqref{warning2}  and choosing $\alpha = \sqrt K$, we conclude
 \begin{equation}\label{warning4}
H(2r_2) \le e^C \left(\frac{r_3}{2r_2}\right)^{C' \sqrt K} H(r_3)^\frac{\beta_0}{\alpha_0 + \beta_0} H(r_1)^\frac{\alpha_0}{\alpha_0 + \beta_0}.
\end{equation}

We now consider the quantity
\begin{equation}\label{h}
h(r) = \int_{B_r} u^2 \mu.
\end{equation}
The following estimates are easily verified from \eqref{h1} and \eqref{h}
\[
H(r) \leq r^{2\alpha} h(r),\ \ \ \ \  \text{and}\ \ \ \ \ h(r) \leq \frac{H(s)}{(s^2 - r^2)^{\alpha}},\  0 < r < s < R_1.
\]
From these estimates and \eqref{warning4} we obtain
\begin{equation}\label{3ball}
h(r_2) \le e^C  (\frac{r_3}{2r_2})^{C''\sqrt{K}}  h(r_3)^\frac{\beta_0}{\alpha_0 + \beta_0} h(r_1)^\frac{\alpha_0}{\alpha_0 + \beta_0},
\end{equation}
for $r_1 <r_2 <2r_2 < r_3 < R_1$. At this point, we take $r_2= \frac{R_1}{3}$, $r_3=R_1$. If  
\[
C_0 = ||u||_{L^\infty(B_{R_1})}^2 \int_{B_{R_1}} \mu > 0,
\]
 then we clearly have $h(R_1) \le C_0$, and we conclude from \eqref{3ball} that
\begin{equation}\label{3ballbis}
h(R_1/3)^{1 + \frac{\beta_0}{\alpha_0}} \le e^{C(1 + \frac{\beta_0}{\alpha_0})}  \left(\frac{3}{2}\right)^{C''(1+\frac{\beta_0}{\alpha_0})\sqrt{K}}  C_0^\frac{\beta_0}{\alpha_0} h(r),\ \ \ \ 0<r<R_1/3.
\end{equation}
If we set $A = e^C$ and $\gamma =  \frac{\overline C^{2}}{\log(3/2)}$, then $q = \beta_0/\alpha_0 = - \log(r/R_1)^\gamma - \overline C^{2}$, and recalling that $\overline C \ge 1$ we obtain from \eqref{3ballbis} for $0<r<R_1/3$
\[
h(r)  \ge C_0 \left(\frac{h(R_1/3)}{AC_0}\right)^{1 + q} \left(\frac{3}{2}\right)^{-C''(1+q)\sqrt{K}}
 \ge C_0 M_0^{1+q} \left(\frac{r}{R_1}\right)^{B\sqrt K},
\]
where we have let $M_0 = \frac{h(R_1/3)}{AC_0}$, and $B = \gamma C'' \log(3/2)$. 
If $M_0\ge 1$ this estimate implies in a trivial way for $0<r < R_1/3$
\[
h(r) \geq C_0  \left(\frac{r}{R_1}\right)^{B\sqrt K}.
\]
If instead $0< M_0\le 1$, keeping in mind that $\overline C \ge 1$, with $B' = \max\{B,\gamma \log(1/M_0)\}$ we obtain for $0<r < R_1/3$
\[
h(r) \geq C_0  \left(\frac{r}{R_1}\right)^{B\sqrt K + \gamma \log(1/M_0)}\ge C_0   \left(\frac{r}{R_1}\right)^{B' (1+\sqrt K)} \ge C_0   \left(\frac{r}{R_1}\right)^{2 B' \sqrt K},
\]
where the last inequality follows by remembering that $K \ge 1$.
In either case, the desired conclusion of Theorem \ref{main} follows by noticing that $h(r) \leq ||u||^2 _{L^{\infty} (B_r)} \int_{B_r} \mu$, and that $\int_{B_r} \mu \le \la^{-1} \int_{B_r} \psi = \la^{-1} \omega r^Q$, where we have let $\omega = \int_{B_1} \psi$. In fact, we would find 
\[
||u||_{L^{\infty} (B_r)} \ge C_3   \left(\frac{r}{R_1}\right)^{C_4 \sqrt K},
\]
with $C_3 = C_0 \sqrt{\frac{\la}{\omega R_1^Q}}$ and $C_4 = 2 B'$. This finishes the proof of Theorem \ref{main}.

\end{document}